\documentclass[a4paper,11pt]{article}
\usepackage{amsmath,indentfirst,amsfonts,dsfont,amssymb,amsthm,graphicx,enumerate,mathrsfs,color}
\usepackage{algorithm}               
\usepackage{algorithmic}             

\newtheorem{definition}{Definition}
\newtheorem{theorem}{Theorem}
\newtheorem{lemma}{Lemma}

\newtheorem{remark}{Remark}

\newtheorem{conjecture}{Conjecture}
\newtheorem{question}{Question}
\title{\bf  Positive Definite Tensors to Nonlinear Complementarity Problems}
\author{
Maolin Che\thanks{E-mail: 14110180046@fudan.edu.cn and chnmaolinche@gmail.com. School of
Mathematical Sciences, Fudan University, Shanghai, 200433, P. R. of
China.  This author is supported by the National Natural Science
Foundation of China under grant 11271084.} \and Liqun Qi \thanks{ E-mail: maqilq@polyu.edu.hk. Department of Applied Mathematics, The Hong Kong Polytechnic University, Hung Hom, Kowloon,
Hong Kong. This author's work was supported by the Hong Kong Research
Grant Council (Grant No. PolyU 502510, 502111, 501212, 501913).} \and Yimin Wei\thanks{
Corresponding author (Y. Wei).
E-mail: ymwei@fudan.edu.cn and yimin.wei@gmail.com. School of
Mathematical Sciences and Shanghai Key Laboratory of Contemporary
Applied Mathematics, Fudan University, Shanghai, 200433, P. R. of
China. This author is supported by the National Natural Science
Foundation of China under grant 11271084.}}
\date{\today}

\begin{document}
\maketitle
\begin{abstract}
The main purpose of this note is to investigate some kinds of nonlinear complementarity problems (NCP). For the structured tensors, such as, symmetric positive definite tensors and copositive tensors, we derive the existence theorems on a solution of these kinds of nonlinear complementarity problems. We prove that a unique solution of the NCP exists under the condition of  diagonalizable tensors.
  \bigskip


  \bigskip

\end{abstract}
\newpage
\section{Introduction}
Let $F$ be a mapping from $\mathds{R}^{n}$ into itself. The nonlinear complementarity problem, denoted by $\mathrm{NCP}(F)$, is to find a vector $\mathbf{x}^{*}\in \mathds{R}_{+}^{n}$ such that
\begin{equation*}
F(\mathbf{x}^{*})\in \mathds{R}_{+}^{n},\quad F(\mathbf{x}^{*})^{\top}\mathbf{x}=0.
\end{equation*}
When $F(\mathbf{x})$ is an affine function of $\mathbf{x}$, say $F(\mathbf{x})=\mathbf{q}+M\mathbf{x}$ for some given vectors $\mathbf{q}\in \mathds{R}^{n}$ and matrix $M\in\mathds{R}^{n\times n}$, the problem $\mathrm{NCP}(F)$ reduces to the linear complementarity problem, which is denoted by $\mathrm{LCP}(\mathbf{q},M)$. The results of the linear complementarity problem can be found in the references (\cite{Berman, lcp}).

A further generalization of the $\mathrm{NCP}(F)$ is the variational inequality problem: given a mapping $F:\mathds{R}^{n}\rightarrow\mathds{R}^{n}$, and $\emptyset\neq \mathfrak{K}\subseteq\mathds{R}^{n}$, find a $\mathbf{x}^{*}\in \mathfrak{K}$ satisfying
\begin{equation*}\label{e1}
(\mathbf{y}-\mathbf{x}^{*})^{\top}F(\mathbf{x}^{*})\geq0, \quad\text{for all }\mathbf{y}\in \mathfrak{K},
\end{equation*}
denoted by $\mathrm{VI}(\mathfrak{K},F)$.

If $\mathfrak{K}=\{\mathbf{x}:\mathbf{x}\geq {\bf 0}\}$, then a $\mathbf{x}^{*}$ is a solution of  $\mathrm{VI}(\mathfrak{K},F)$, solves the $\mathrm{NCP}(F)$.

Over the past decade, the research of finite-dimensional variational inequality and complementarity problems (\cite{ncpv1, ncpv2, NCP,cp, strict, Jacobians_ncp,NCPalgorithm}) has been rapidly developed in the theory of existence, uniqueness and sensitivity of solutions, theory of algorithms, and the application of these techniques to transportation planning, regional science, socio-economic analysis, energy modeling and game theory.

Qi \cite{Qiliqun}  defined two kinds of eigenvalues
and described some relative results similar to the matrix
eigenvalues. Lim \cite{Lim}  proposed another definition
of eigenvalues, eigenvectors, singular values, and singular vectors
for tensors based on a constrained variational approach, much like
the Rayleigh quotient for symmetric matrix eigenvalues, independently.

It is well-known that $A$ is a $P$ matrix (\cite{Berman}), if and only if the linear complementarity
problem,
$$
\rm{find}~~ z \in \mathds{R}^n
~~\rm{such~ that}~~ z \geq  \mathbf{0}, \quad q + Az \geq \mathbf{0}, ~~\rm{and}~~ z^{\top}(q + Az) = 0.
$$
has a unique solution for all $q \in \mathds{R}^n$. Then for a $P$ tensor (\cite{SQ}) $\mathcal{A}\in T_{m,n}, (m > 2)$, does a similar
property hold for the following nonlinear complementarity problem,
$$
\rm{find}~~ x \in \mathds{R}^n
~~\rm{such~ that}~~ x \geq  \mathbf{0}, \quad q + \mathcal{A} x^{m-1} \geq \mathbf{0}, ~~\rm{and}~~ x^{\top}(q + \mathcal{A} x^{m-1} ) =0?
$$

In this note, we will consider a special kind of $\mathrm{NCP}(F)$, that is, $F_{i}(\mathbf{x})$ is a multivariate polynomial and the degree of $F_{i}(\mathbf{x})$ is $k_{i}$, then $F(\mathbf{x})$ can be expressed by,
\begin{equation*}
F(\mathbf{x})=\sum\limits_{i=1}^{k}\mathcal{A}_{i}\mathbf{x}^{i-1},
\end{equation*}
where $\mathcal{A}_{i}\in T_{i-1,n}$, $\mathcal{A}_{i}\mathbf{x}^{i-1}$ means the tensor-vector product given in Section 2 and $k=\max\limits_{1\leq i\leq n}k_{i}$. Specially, $\mathcal{A}_{1}$ is a vector and $\mathcal{A}_{2}$ is a matrix.

Throughout this paper, we assume that $m,\ n\ ( \geq 2)$ are positive integers and
$m$ is even. We use small letters $x,u,v,\dots,$ for scalers,
small bold letters $\mathbf{x},\mathbf{u},
\mathbf{v},\dots,$ for vectors, capital letters
$A,B,C,\dots,$ for matrices, calligraphic letters
$\mathcal{A},\mathcal{B},\mathcal{C},\dots,$ for tensors, and $\mathfrak{A},\mathfrak{B},\mathfrak{C},\dots,$ for the subsets in $\mathds{R}^{n}$. Denote
$[n]=\{1,2,\dots,n\}$. $\mathbf{0}$ means a column vector in $\mathds{R}^{n}$, where its all entries are zeros. $\mathds{R}_{+}^{n}$ denotes the nonnegative orthant of $\mathds{R}^{n}$.

Let $\mathds{R}$ be a real field. An order $m$ dimension $n$ tensor $\mathcal{A}$ consists of $n^{m}$ real entries. Let $T_{m,n}$ denote all order $m$ dimension $n$ tensors where every entry of each tensor in $T_{m,n}$ is in $\mathds{R}$. $\mathcal{D}\in T_{m,n}$ is diagonal if all off-diagonal entries are zero. Particularly, when the diagonal entries of $\mathcal{D}$ are 1, then $\mathcal{D}$ is called the identity tensor and denote it as $\mathcal{I}$.

The rest of
this note is organized as follows. Section 2 introduces definitions
of basic preliminaries and some results of the nonlinear complementarity problems. The main result in Section 3 is to consider the solvability of Questions \ref{q1} and \ref{q2}.
 We  conclude our
paper in Section 4.
\section{Preliminaries}
In this section, we present five definitions about tensors and nonlinear mappings, four lemmas associated with these definitions and the solutions of the nonlinear complementarity problems, and two questions proposed as follows.
\subsection{Definitions}
The mode-$k$ tensor-matrix product and mode-$k$ tensor-vector
product  of $\mathcal{A}\in T_{m,n}$ are
defined as follows.
\begin{definition}{\bf (\cite{Kolda})}\label{d1} The mode-$k$ product of a tensor $\mathcal{A}\in T_{m,n}$ and  a matrix $B\in \mathds{R}^{n\times n}$, denoted
by $\mathcal{A}\times_{k}B$, is a tensor $\mathcal{C}\in T_{m,n}$ of
which are given by
$$\mathcal{C}_{i_{1}\dots i_{k-1}ji_{k+1}\dots i_{m}}=\sum_{i_{k}=1}^{n}\mathcal{A}_{i_{1}\dots i_{k-1}i_{k}i_{k+1}\dots i_{m}}b_{ji_{k}},
\quad k\in [m].$$
Particularly, the mode-$k$ multiplication of a tensor $\mathcal{A}\in T_{m,n}$ and a vector $\mathbf{x}\in \mathds{R}^{n}$ is denoted by
$\mathcal{A}\bar{\times}_{k}\mathbf{x}$. Set $\mathcal{C}=\mathcal{A}\bar{\times}_{k}\mathbf{x}$, then, element-wise, we have
\begin{equation*}
\mathcal{C}_{i_{1}\dots i_{k-1}i_{k+1}\dots i_{m}}=\sum\limits_{i_{k}=1}^{n}\mathcal{A}_{i_{1}\dots i_{k-1}i_{k}i_{k+1}\dots i_{m}}x_{i_{k}}.
\end{equation*}
\end{definition}
According to Definition \ref{d1}, let $m$ vectors $\mathbf{x}_{k}\in \mathds{R}^{n}$, the formula
$\mathcal{A}\bar{\times}_{1}\mathbf{x}_{1}\dots\bar{\times}_{m}\mathbf{x}_{m}$ is easy to define.
If these $m$ vectors are the same vector, denoted by $\mathbf{x}$, then $\mathcal{A}\bar{\times}_{1}\mathbf{x}\dots\bar{\times}_{m}\mathbf{x}$ can be simplified as $\mathcal{A}\mathbf{x}^{m}$.

Our next definition is motivated by the class of copositive matrices (\cite{Berman}), which in turn generalizes that of nonnegative matrices.
\begin{definition}{\bf (\cite{NCP})}\label{d2}
A mapping $F:\mathfrak{X} \rightarrow\mathds{R}^{n}$ is said to be
\begin{enumerate}
\item[\rm (a)] Copositive with respect to $\mathfrak{X}$, if
\begin{equation*}
[F(\mathbf{x})-F(\mathbf{0})]^{\top}\mathbf{x}\geq 0,\quad \forall\mathbf{x}\in \mathfrak{X}.
\end{equation*}
\item[\rm(b)] Strictly copositive with respect to $\mathfrak{X}$, if
\begin{equation*}
[F(\mathbf{x})-F(\mathbf{0})]^{\top}\mathbf{x}> 0,\quad \forall\mathbf{x}\in \mathfrak{X},\ \mathbf{x}\neq\mathbf{0}.
\end{equation*}
\item[\rm (c)] Strongly copositive with respect to $\mathfrak{X}$, if there exists a positive scalar $\alpha$ such that
\begin{equation*}
[F(\mathbf{x})-F(\mathbf{0})]^{\top}\mathbf{x}\geq \alpha\|\mathbf{x}\|_{2}^{2},\quad \forall\mathbf{x}\in \mathfrak{X}.
\end{equation*}
\end{enumerate}
\end{definition}
The definition of a symmetric tensor (\cite{Lim, Qiliqun}) is stated as follows.
\begin{definition}\label{d3} Suppose that $\mathcal{A}\in T_{m,n}$. $\mathcal{A}$ is called symmetric if $\mathcal{A}_{i_{1}i_{2}\dots i_{m}}$ is
invariant by any permutation $\pi$, that is,
$\mathcal{A}_{i_{1}i_{2}\dots
i_{m}}=\mathcal{A}_{\pi(i_{1},i_{2},\dots,i_{m})}$ where all
$i_{k}\in[n]$ with $k\in[m]$.
We denote all symmetric tensors by
$ST_{m,n}$.
\end{definition}
When $m$ is even and $\mathcal{A}$ is symmetric, we say that
\begin{enumerate}
\item[\rm (a)]
$\mathcal{A}$ is positive definite (\cite{Qiliqun}), if $\mathcal{A}\mathbf{x}^{m}>0$ for all $\mathbf{x}\in\mathds{R}^{n}$ and $\mathbf{x}\neq\mathbf{0}$,
\item[\rm (b)]
 $\mathcal{A}$ is a copositive tensor (\cite{Copositive}), if for any $x\in\mathds{R}_{+}^{n}$, then $\mathcal{A}\mathbf{x}^{m}\geq 0$,
\item[\rm (c)]
   $\mathcal{A}$ is a strictly copositive tensor (\cite{Copositive}), if for any $x\in\mathds{R}_{+}^{n}$, then $\mathcal{A}\mathbf{x}^{m}> 0$.
\end{enumerate}
    The set of all positive definite tensors is denoted by $SPT_{m,n}$.

The mapping
\begin{equation}\label{e2}
G(\mathbf{x})=F(\mathbf{x})-F(\mathbf{0})
 \end{equation}
 plays an important role in the nonlinear complementarity problem (\cite{NCP}), this again is motivated by the linear complementarity problem.

The strict copositivity of a mapping can be relaxed through the introduction of the class of d-regular mappings.
\begin{definition} {\bf (\cite{NCP})}\label{d4}
For any vector $\mathbf{x}\in\mathds{R}_{+}^{n}$, we define the index sets,
\begin{equation*}
I_{+}(\mathbf{x})=\{i:x_{i}>0\}\quad\text{and}\quad I_{0}(\mathbf{x})=\{i:x_{i}=0\}.
\end{equation*}
Let $\mathbf{d}>\mathbf{0}$ be an arbitrary vector in $\mathds{R}^{n}$. A mapping $G:\mathds{R}^{n}\rightarrow\mathds{R}^{n}$ is said to be d-regular,
 if the following system of equations has no solution in $(\mathbf{x},t)\in\mathds{R}_{+}^{n}\times\mathds{R}_{+}$ with $\mathbf{z}\neq\mathbf{0}$,
\begin{equation}\label{e3}
\begin{split}
G_{i}(\mathbf{x})+td_{i}&=0,\quad i\in I_{+}(\mathbf{x}),\\
G_{i}(\mathbf{x})+td_{i}&\geq0,\quad i\in I_{0}(\mathbf{x}).
\end{split}
\end{equation}
Equivalently, $G$ is d-regular if, for any scalar $r>0$, the augmented nonlinear complementarity problem $NCP(H)$ defined by $H:\mathds{R}^{n+1}\rightarrow\mathds{R}^{n+1}$,
\begin{equation*}
H\begin{pmatrix}\mathbf{x}\\ t \end{pmatrix}=\begin{pmatrix}G(\mathbf{x})+t\mathbf{d}\\ r-\mathbf{d}^{\top}\mathbf{x}\end{pmatrix},
\end{equation*}
has no solution $(\mathbf{x},t)$ with $\mathbf{x}\neq\mathbf{0}$.
\end{definition}
Similar to the diagonalizable matrices (\cite{Golub}), the definition of the diagonalizable tensors is presented as follows.
\begin{definition}\label{d5}
Suppose that $\mathcal{A}\in ST_{m,n}$. $\mathcal{A}$ is called diagonalizable if $\mathcal{A}$ can be represented as
\begin{equation*}
\{\mathcal{A}\in T_{m,n}|\mathcal{A}=\mathcal{D}\times_{1}B\times_{2}B\dots\times_{m}B\},
\end{equation*}
where $B\in\mathds{R}^{n\times n}$ with $\det(B)\neq 0$ and $\mathcal{D}$ is a diagonal tensor. Denote all diagonalizable tensors by $D_{m,n}$.
\end{definition}
\subsection{Lemmas}
The following lemma is an existence and uniqueness theorem by Cottle (\cite{Jacobians_ncp}). It involves the notion of positively bounded Jacobians, and the original proof was constructive, in the sense that an algorithm was employed to actually compute the unique solution.
\begin{lemma}{\bf (\cite{Jacobians_ncp, NCP})}\label{l1}
Let $F:\mathds{R}_{+}^{n}\rightarrow\mathds{R}^{n}$ be continuously differentiable and suppose that there exists one $\delta\in(0,1)$, such that all principal minors of the Jacobian matrix $\nabla F(\mathbf{x})$ are bounded between $\delta$ and $\delta^{-1}$, for all $\mathbf{x}\in\mathds{R}_{+}^{n}$. Then the $\mathrm{NCP}(F)$ has a unique solution.
\end{lemma}
If mapping $F$ is strictly copositive, then the following result holds.
\begin{lemma}{\bf (\cite{strict})}\label{l2}
Let $F:\mathds{R}_{+}^{n}\rightarrow\mathds{R}^{n}$ be continuous and strictly copositive with respect to $\mathds{R}_{+}^{n}$. If there exists a mapping $c:\mathds{R}_{+}\rightarrow\mathds{R}$ such that $c(\lambda)\rightarrow\infty$ as $\lambda\rightarrow\infty$, and for all $\lambda\geq 1$, $\mathbf{x}\geq\mathbf{0}$,
\begin{equation}\label{e4}
[F(\lambda\mathbf{x})-F(\mathbf{0})]^{\top}\mathbf{x}\geq c(\lambda)[F(\mathbf{x})-F(\mathbf{0})]^{\top}\mathbf{x},
\end{equation}
then the problem $\mathrm{NCP}(F)$ has a nonempty, compact solution set.
\end{lemma}
For the linear complementarity problem, the mapping $G$, given in formula (\ref{e2}), is obviously linear and thus, condition (\ref{e4}) is satisfied with $c(\lambda)=\lambda$. More generally, the same condition will hold with $c(\lambda)=\lambda^{\alpha}$, if $G$ is positively homogeneous of degree $\alpha>0$; i.e., if $G(\lambda\mathbf{x})=\lambda^{\alpha}G(\mathbf{x})$ for $\lambda>0$.

If $F$ is strictly copositive with respect to $\mathds{R}_{+}^{n}$, then the mapping $G$ in (\ref{e2}) is d-regular for any $\mathbf{d}>\mathbf{0}$. The following lemma presents an existence result for the nonlinear complementarity problem with d-regular mapping.
\begin{lemma}{\bf (\cite{cp})}\label{l3}
Let $F$ be a continuous mapping from $\mathds{R}^{n}$ into itself and  $G$ defined by (\ref{e2}). Suppose that $G$ is positively homogeneous of degree $\alpha>0$ and that $G$ is d-regular for some $\mathbf{d}>\mathbf{0}$. Then the problem $\mathrm{NCP}(F)$ has a nonempty, compact solution set.
\end{lemma}
The main characterization theorem for copositive tensors can be summarized as follows.
\begin{lemma}{\bf (\cite[Theorem 5]{Copositive})}\label{l4}
Let $\mathcal{A}\in T_{m,n}$ be a symmetric tensor. Then, $\mathcal{A}$ is copositive of and only if
\begin{equation*}
\min\{\mathcal{A}\mathbf{x}^{m}:\mathbf{x}\in\mathds{R}_{+}^{n},\ \sum_{i=1}^{n}x_{i}^{m}=1\}\geq 0.
\end{equation*}
$\mathcal{A}$ is strictly copositive if and only if
\begin{equation*}\min\{\mathcal{A}\mathbf{x}^{m}:\mathbf{x}\in\mathds{R}_{+}^{n},\ \sum\limits_{i=1}^{n}x_{i}^{m}=1\}> 0.
\end{equation*}
\end{lemma}
\subsection{Problem Description}
In this subsection, we propose two questions which we shall discuss in this note.
\begin{question}{\bf (\cite{SQ})}\label{q1}
Given $\mathcal{A}\in T_{m,n}$ and $\mathbf{q}\in \mathds{R}^{n}$. The $\mathrm{NCP}(\mathbf{q},\mathcal{A})$ is to find a vector $\mathbf{x}\in \mathds{R}_{+}^{n}$ such that
\begin{equation*}
\mathcal{A}\mathbf{x}^{m-1}+\mathbf{q}\in \mathds{R}_{+}^{n},\quad \mathcal{A}\mathbf{x}^{m}+\mathbf{q}^{\top}\mathbf{x}=0.
\end{equation*}
\end{question}
\begin{question}\label{q2}
Given $\mathcal{A}_{k}\in T_{m-(2k-2),n}$ and $\mathbf{q}\in \mathds{R}^{n}$ with $k=1,2,\dots,m/2$. The $\mathrm{GNCP}(\mathbf{q},\{\mathcal{A}_{k}\})$ is to find a vector $\mathbf{x}\in \mathds{R}_{+}^{n}$ such that
\begin{equation*}
\sum_{k=1}^{m/2}\mathcal{A}_{k}\mathbf{x}^{m-(2k-1)}+\mathbf{q}\in \mathds{R}_{+}^{n},\quad \sum_{k=1}^{m/2}\mathcal{A}_{k}\mathbf{x}^{m-2k+2}+\mathbf{q}^{\top}\mathbf{x}=0,
\end{equation*}
where $\mathcal{A}_{m/2}$ is a square matrix.
\end{question}
Let $\mathrm{FEA}(\mathbf{q},\mathcal{A})=\{\mathbf{x}\in \mathds{R}_{+}^{n}:\mathcal{A}\mathbf{x}^{m-1}+\mathbf{q}\in \mathds{R}_{+}^{n}\}$. If $\mathrm{FEA}(\mathbf{q},\mathcal{A})\neq\emptyset$, then we see that $\mathrm{NCP}(\mathbf{q},\mathcal{A})$ is feasible.
It is obvious that Question \ref{q1} is a special case of Question \ref{q2}. However, for simplicity, we only consider the solvability of Question \ref{q1}, and make use of the results obtained by solving Question \ref{q1}, we then consider the solvability of Question \ref{q2}.
\section{Main results}
Without loss of generality, suppose that $\mathbf{q}\in\mathds{R}^{n}$ in Questions \ref{q1} and \ref{q2} is nonzero. For example, let $\mathcal{A}\in ST_{m,n}$ be positive definite. If $\mathbf{q}$ is zero, then the solution of Question \ref{q1} is zero. This situation is extraordinary,  in order to avoid this situation, let $\mathbf{q}\in\mathds{R}^{n}$ in Questions \ref{q1} and \ref{q2} be nonzero.
\subsection{Necessary conditions for Solving Question \ref{q1}}
The cornerstone for the {\it necessary} conditions to be presented is the nonlinear programming formulation of the Question \ref{q1},
\begin{equation}\label{e5}
\begin{split}
\min&\qquad\mathcal{A}\mathbf{x}^{m}+\mathbf{q}^{\top}\mathbf{x}\\
s.t.&\qquad\mathcal{A}\mathbf{x}^{m-1}+\mathbf{q}\in \mathds{R}_{+}^{n},\ \mathbf{x}\in \mathds{R}_{+}^{n}.
\end{split}
\end{equation}
Because $\mathrm{FEA}(\mathbf{q},\mathcal{A})$ is also the feasible set of (\ref{e5}), if $\mathbf{x}_{*}$ minimizes the nonlinear programming given in (\ref{e5}) and $\mathcal{A}\mathbf{x}_{*}^{m}+\mathbf{q}^{\top}\mathbf{x}_{*}=0$, then $\mathbf{x}_{*}$ is a solution of Question \ref{q1}. According to first-order necessary conditions given in \cite{numerical}, we obtain the following result.
\begin{theorem}\label{t1}
If $\mathrm{FEA}(\mathbf{q},\mathcal{A})\neq\emptyset$ and $\mathbf{x}_{*}$ is a local solution of (\ref{e5}). Then, there exists a vector $\mathbf{u}_{*}$ of multipliers satisfying the conditions,
\begin{equation}\label{e6}
\begin{split}
\mathbf{q}+m\mathcal{A}\mathbf{x}_{*}^{m-1}-(m-1)\mathcal{A}\mathbf{x}_{*}^{m-2}\mathbf{u}_{*}&\geq\mathbf{0}\\
\mathbf{x}_{*}^{\top}(\mathbf{q}+m\mathcal{A}\mathbf{x}_{*}^{m-1}-(m-1)\mathcal{A}\mathbf{x}_{*}^{m-2}\mathbf{u}_{*})&=0\\
\mathbf{u}_{*}&\geq\mathbf{0}\\
\mathbf{u}_{*}^{\top}(\mathbf{q}+\mathcal{A}\mathbf{x}_{*}^{m-1})&=0.
\end{split}
\end{equation}
Finally, the vectors $\mathbf{x}_{*}$ and $\mathbf{u}_{*}$ satisfy
\begin{equation}\label{e7}
(m-1)(\mathbf{x}_{*}-\mathbf{u}_{*})_{i}(\mathcal{A}\mathbf{x}_{*}^{m-2}(\mathbf{x}_{*}-\mathbf{u}_{*}))_{i}\leq0, \quad i\in[n].
\end{equation}
\end{theorem}
\begin{proof}
Since $\mathrm{FEA}(\mathbf{q},\mathcal{A})\neq\emptyset$,  the nonlinear program (\ref{e5}) is feasible. Such an optimal solution $\mathbf{x}_{*}$ and a suitable vector $\mathbf{u}_{*}$ of multipliers will satisfy the Karush-Kuhn-Tucker conditions (\ref{e6}). To prove (\ref{e7}), we examine the inner product $$\mathbf{x}_{*}^{\top}(\mathbf{q}+m\mathcal{A}\mathbf{x}_{*}^{m-1}-(m-1)\mathcal{A}\mathbf{x}_{*}^{m-2}\mathbf{u}_{*})=0,$$
at the componentwise level and deduce that for all $i\in[n]$,
\begin{equation}\label{e8}
(m-1)(\mathbf{x}_{*})_{i}(\mathcal{A}\mathbf{x}_{*}^{m-2}(\mathbf{x}_{*}-\mathbf{u}_{*}))_{i}\leq 0,
\end{equation}
using the fact that $\mathbf{x}_{*}\in \mathrm{FEA}(\mathbf{q},\mathcal{A})$. Similarly, multiplying the $i$th component in $$\mathbf{q}+m\mathcal{A}\mathbf{x}_{*}^{m-1}-(m-1)\mathcal{A}\mathbf{x}_{*}^{m-2}\mathbf{u}_{*}\geq\mathbf{0},$$ by $\mathbf{u}_{*}$ and then invoking the complementarity condition $$(\mathbf{u}_{*})_{i}(\mathbf{q}+\mathcal{A}\mathbf{x}_{*}^{m-1})_{i}=0,$$ which is implied by $\mathbf{u}_{*}\geq\mathbf{0}$, $\mathbf{u}_{*}^{\top}(\mathbf{q}+\mathcal{A}\mathbf{x}_{*}^{m-1})=0$, and the feasibility of of $\mathbf{x}_{*}$, we obtain
\begin{equation}\label{e9}
-(m-1)(\mathbf{u}_{*})_{i}(\mathcal{A}\mathbf{x}_{*}^{m-2}(\mathbf{x}_{*}-\mathbf{u}_{*}))_{i}\leq0.
\end{equation}
Now, (\ref{e7}) follows by adding (\ref{e8}) and (\ref{e9}).
\end{proof}
\begin{remark}
Cottle (\cite{Jacobians_ncp}) obtained the general matrix case of Theorem \ref{t1}.
\end{remark}
With the help of Theorem \ref{t1}, we prove the following existence result for the $\mathrm{NCP(\mathbf{q},\mathcal{A})}$.
\begin{theorem}\label{t2}
Let nonzero $\mathbf{x}_{*}$ be a local solution of $($\ref{e5}$)$. If $\mathcal{A}\mathbf{x}^{m-2}$ is positive definite for all nonzero $\mathbf{x}\in\mathds{R}^{n}$, then $\mathbf{x}_{*}$ solves $\mathrm{NCP}(\mathbf{q},\mathcal{A})$.
\end{theorem}
\begin{proof}
According to Theorem \ref{t1}, there exists a nonnegative vector $\mathbf{u}_{*}$ such that
\begin{equation*}
(m-1)(\mathbf{x}_{*}-\mathbf{u}_{*})_{i}(\mathcal{A}\mathbf{x}_{*}^{m-2}(\mathbf{x}_{*}-\mathbf{u}_{*}))_{i}\leq0, \quad i\in[n],
\end{equation*}
that is,
\begin{equation*}
(\mathbf{x}_{*}-\mathbf{u}_{*})^{\top}(\mathcal{A}\mathbf{x}_{*}^{m-2}(\mathbf{x}_{*}-\mathbf{u}_{*}))\leq0.
\end{equation*}
According to proposition assumptions, we know that $\mathbf{x}_{*}=\mathbf{u}_{*}$. Based on (\ref{e6}), then, $\mathbf{x}_{*}$ solves $\mathrm{NCP}(\mathbf{q},\mathcal{A}))$.
\end{proof}
Moreover, we can derive some results about Question \ref{q2}, similar to Theorems \ref{t1} and \ref{t2}. Here, we do not list them out.
\subsection{Solving Question \ref{q1}}
In Question \ref{q1}, let $F(\mathbf{x})=\mathcal{A}\mathbf{x}^{m-1}+\mathbf{q}$. We first consider some properties of $F(\mathbf{x})$ when $\mathcal{A}$ is selected from sets of structured tensors.
\begin{theorem}\label{t3}
Suppose $\mathcal{A}\in ST_{m,n}$ and $\mathbf{x}\in\mathds{R}_{+}^{n}$.
\begin{enumerate}
\item[\rm (a)] If $\mathcal{A}$ is (strictly) copositive, then the mapping $F(\mathbf{x})$ is (strictly) copositive with respect to $\mathds{R}_{+}^{n}$,
\item[\rm (b)] If $\mathcal{A}$ is positive definite, then the mapping $F(\mathbf{x})$ is strongly copositive with respect to $\mathds{R}_{+}^{n}$ when $\alpha\leq\lambda_{\min}\ (\leq\lambda_{\min}\frac{\|\mathbf{x}\|_{2}^{2}}{\|\mathbf{x}\|_{m}^{m}})$, where $\lambda_{\min}$ is the smallest Z-eigenvalue (H-eigenvalue) of $\mathcal{A}$.
\end{enumerate}
\end{theorem}
\begin{proof} According to Definition \ref{d2}, because of $F(\mathbf{x})=\mathcal{A}\mathbf{x}^{m-1}+\mathbf{q}$, so $[F(\mathbf{x})-F(\mathbf{0})]^{\top}\mathbf{x}=\mathcal{A}\mathbf{x}^{m}$. Since $\mathcal{A}$ is (strictly) copositive,  $\mathcal{A}\mathbf{x}^{m}\ (>)\ \geq 0$. Then, part $(a)$ is proved. We now prove part $(b)$. Since $\mathcal{A}$ is positive definite,  according to \cite[Theorem 5]{Qiliqun}, we know that the smallest Z-eigenvalue (H-eigenvalue) of $\mathcal{A}$ is greater than zero. Hence, part $(b)$ is proved.
\end{proof}
When $\mathcal{A}\in D_{m,n}$ is positive definite, the following theorem will give a porperty of the Jacobian matrix $\nabla F(\mathbf{x})$, where $\mathbf{x}$ is nonzero vector.
\begin{theorem}\label{t4}
Let $\mathcal{A}\in SD_{m,n}$ be positive definite. Then the Jacobian matrix $\nabla F(\mathbf{x})$ is positive definite with $\mathbf{x}\neq \mathbf{0}$.
\end{theorem}
\begin{proof}
As $\mathcal{A}$ is diagonalizable, for a vector $\mathbf{x}$, according to Definition \ref{d1}, we have
\begin{equation*}
\begin{split}
\mathcal{A}\mathbf{x}^{m}&=(\mathcal{D}\times_{1}B\times_{2}B\dots\times_{m}B)\mathbf{x}^{m}=\mathcal{D}(B^{\top}\mathbf{x})^{m}\\
&=\mathcal{D}\mathbf{y}^{m}\ (\mathbf{y}\overset{\triangle}{=}B^{\top}\mathbf{x})=\sum_{i=1}^{n}d_{i}y_{i}^{m},
\end{split}
\end{equation*}
where $d_{i}$ is the $i$th diagonal entry of $\mathcal{D}$. According to the proposition assumption, $d_{i}>0$,  we have $\mathcal{A}\mathbf{x}^{m}>0$ for all nonzero vectors $\mathbf{x}$.

Since the Jacobian matrix $\nabla F(\mathbf{x})$ is $(m-1)\mathcal{A}\mathbf{x}^{m-2}$, for any nonzero vector $\mathbf{z}\in \mathds{R}^{n}$, $\mathbf{z}^{\top}\nabla F(\mathbf{x})\mathbf{z}$ can be expressed by
$$\mathbf{z}^{\top}\nabla F(\mathbf{x})\mathbf{z}=(m-1)\sum\limits_{i=1}^{n}d_{i}y_{i}^{m-2}z_{i}^{2}>0.$$
Hence, the Jacobian matrix $\nabla F(\mathbf{x})$ is positive definite with $\mathbf{x}\neq \mathbf{0}$.
\end{proof}
\begin{theorem}\label{t5}
Suppose that $\mathcal{A}\in ST_{m,n}$. For Question \ref{q1}, the following results hold.
\begin{enumerate}
\item[\rm (a)] If $\mathcal{A}$ is diagonalizable and positive definite, then the $\mathrm{NCP}(\mathbf{q},\mathcal{A})$ has a unique solution,
\item[\rm (b)] If $\mathcal{A}$ is positive definite, then the $\mathrm{NCP}(\mathbf{q},\mathcal{A})$  has a nonempty, compact solution set,
\item[\rm (c)] If $\mathcal{A}$ is strictly copositive with respect to $\mathds{R}_{+}^{n}$, then the $\mathrm{NCP}(\mathbf{q},\mathcal{A})$  has a nonempty, compact solution set.
\end{enumerate}
\end{theorem}
\begin{proof}
Part $(a)$: it is obvious, according to Lemma \ref{l1} and Theorem \ref{t4}. We will prove part $(b)$ as follows. Since $\mathcal{A}$ is positive definite, according to Theorem \ref{t3}, we have $F(\mathbf{x})=\mathcal{A}\mathbf{x}^{m-1}+\mathbf{q}$ is strictly copositive. Let $c(\lambda)=\lambda^{\alpha}$ with $0<\alpha\leq m$, we know that $c(\lambda)\rightarrow\infty$ as $\lambda\rightarrow\infty$, then, based on Lemma \ref{l2}, we know that if $\mathcal{A}$ is positive definite, then the $\mathrm{NCP}(\mathbf{q},\mathcal{A})$  has a nonempty, compact solution set.

The rest part is to show that  part $(c)$. By Theorem \ref{t3}, we obtain that $F(\mathbf{x})=\mathcal{A}\mathbf{x}^{m-1}+\mathbf{q}$ is strictly copositive. In \cite{cp}, we have that if $F(\mathbf{x})$ is strictly copositive with respect to $\mathds{R}_{+}^{n}$, then the mapping $G$ in (\ref{e2}) is $d$-regular for any $d>0$. Then, according to Lemma \ref{l3}, if $\mathcal{A}$ is strictly copositive with respect to $\mathds{R}_{+}^{n}$, then the $\mathrm{NCP}(\mathbf{q},\mathcal{A})$  has a nonempty, compact solution set.
\end{proof}
\subsection{Solving Question \ref{q2}}
In the above subsection, we consider the solvability of Question \ref{q1}. Analogously, The following theorems have been described by the solvability of Question \ref{q2}.
\begin{theorem}
Suppose that $\mathcal{A}_{k}\in T_{m-(2k-2),n}$, with $k\in[m/2]$. For Question \ref{q2}, the following results hold.
\begin{enumerate}
\item[\rm (a)] If $\mathcal{A}_{k}$ are diagonalizable and positive definite, then the $\mathrm{GNCP}(\mathbf{q},\{\mathcal{A}_{k}\})$ has a unique solution,
\item[\rm (b)] If $\mathcal{A}_{k}$ are positive definite, then the $\mathrm{GNCP}(\mathbf{q},\{\mathcal{A}_{k}\})$  has a nonempty, compact solution set,
\item[\rm (c)] If $\mathcal{A}_{k}$ are strictly copositive with respect to $\mathds{R}_{+}^{n}$, then the $\mathrm{GNCP}(\mathbf{q},\{\mathcal{A}_{k}\})$  has a nonempty, compact solution set,
\end{enumerate}
where $\mathcal{A}_{m/2}$ is a square matrix.
\end{theorem}
It is easy to prove this theorem according to the above description, then, we omit the proof. Based on Lemma \ref{l1}, because $\mathcal{A}_{m/2}$ is a square matrix, so the constraints of matrix $\mathcal{A}_{m/2}$ can be weakened. Hence, a more general result is given as follows.
\begin{theorem}
Suppose that $\mathcal{A}_{k}\in T_{m-(2k-2),n}$, with $k\in[m/2-1]$ and $\mathcal{A}_{m/2}$ is a square matrix. For Question \ref{q2}, the following results hold.
\begin{enumerate}
\item[\rm (a)] If $\mathcal{A}_{k}$ are diagonalizable and positive definite and there exists one $\delta\in(0,1)$, such that all principal minors of $\mathcal{A}_{m/2}$ are bounded between $\delta$ and $\delta^{-1}$, then the $\mathrm{GNCP}(\mathbf{q},\{\mathcal{A}_{k}\})$ has a unique solution,
\item[\rm (b)] If $\mathcal{A}_{k}$ are positive definite and $\mathcal{A}_{m/2}$ is strictly copositive with respect to $\mathds{R}_{+}^{n}$, then the $\mathrm{GNCP}(\mathbf{q},\{\mathcal{A}_{k}\})$  has a nonempty, compact solution set,
\item[\rm (c)] If $\mathcal{A}_{k}\ (k\in[m/2])$ are strictly copositive with respect to $\mathds{R}_{+}^{n}$, then the $\mathrm{GNCP}(\mathbf{q},\{\mathcal{A}_{k}\})$  has a nonempty, compact solution set.
\end{enumerate}
\end{theorem}
\begin{remark}
In the above two theorems, the assumptions of $\mathcal{A}_{k}\in T_{m-(2k-2),n}$, with $k\in[m/2]$ can be appropriately reduced. However, we don't consider these situations.
\end{remark}
\section{Conclusion}
In this note, by some structured tensors, the main task is to consider the existence and uniqueness about the solution of Questions \ref{q1} and \ref{q2}. Now we  present two conjectures about the $\mathrm{NCP}(\mathbf{q},\mathcal{A})$.
\begin{conjecture}
For the part $(\rm a)$ in Theorem \ref{t5}, if $\mathcal{A}\in ST_{m,n}$ is just positive definite, then the $NCP(\mathbf{q},\mathcal{A})$  has a unique solution.
\end{conjecture}
\begin{conjecture}
If $\mathrm{FEA}(\mathbf{q},\mathcal{A})\neq\emptyset$, then the nonlinear program (\ref{e5}) has an optimal solution, $\mathbf{x}_{*}$. Moreover, there exists a vector $\mathbf{u}_{*}$ of multipliers satisfying the conditions,
\begin{equation*}
\begin{split}
\mathbf{q}+m\mathcal{A}\mathbf{x}_{*}^{m-1}-(m-1)\mathcal{A}\mathbf{x}_{*}^{m-2}\mathbf{u}_{*}&\geq\mathbf{0}\\
\mathbf{x}_{*}^{\top}(\mathbf{q}+m\mathcal{A}\mathbf{x}_{*}^{m-1}-(m-1)\mathcal{A}\mathbf{x}_{*}^{m-2}\mathbf{u}_{*})&=0\\
\mathbf{u}_{*}&\geq\mathbf{0}\\
\mathbf{u}_{*}^{\top}(\mathbf{q}+\mathcal{A}\mathbf{x}_{*}^{m-1})&=0.
\end{split}
\end{equation*}
Finally, the vectors $\mathbf{x}_{*}$ and $\mathbf{u}_{*}$ satisfy
\begin{equation*}
(m-1)(\mathbf{x}_{*}-\mathbf{u}_{*})_{i}(\mathcal{A}\mathbf{x}_{*}^{m-2}(\mathbf{x}_{*}-\mathbf{u}_{*}))_{i}\leq0, \quad i\in[n].
\end{equation*}
\end{conjecture}
When $m=2$, this is the theorem about the existence result for a solution of the quadratic program associated with the linear complementarity problem given in \cite{lcp}. Unfortunately, Cottle \cite{Jacobians_ncp} presented some counter examples to explain that this conjecture is not true for the general nonlinear programming.

Finally, for the existence and uniqueness about the solution of Question \ref{q1}, we have an open question given as follows.
\begin{question}
Suppose $\mathcal{A}\in T_{m,n}$ and nonzero $\mathbf{x}\in\mathds{R}^{n}$. What conditions of $\mathcal{A}$ will make sure that there exists one $\delta\in(0,1)$, such that all principal minors of matrix $\mathcal{A}\mathbf{x}^{m-2}$ are bounded between $\delta$ and $\delta^{-1}$, for all $\mathbf{x}\in\mathds{R}_{+}^{n}$?
\end{question}

\newpage

{\small
\bibliographystyle{siam}
\bibliography{specialncp}

\begin{thebibliography}{10}

\bibitem{Berman}
{\sc A.~Berman and R.~J. Plemmons}, {\em Nonnegative {M}atrices in the
  {M}athematical {S}ciences}, vol.~9 of Classics in Applied Mathematics,
  Society for Industrial and Applied Mathematics, Philadelphia, PA, 1994.

\bibitem{Jacobians_ncp}
{\sc R.~W. Cottle}, {\em Nonlinear programs with positively bounded
  {J}acobians}, SIAM J. Appl. Math., 14 (1966), pp.~147--158.

\bibitem{lcp}
{\sc R.~W. Cottle, J.-S. Pang, and R.~E. Stone}, {\em The {L}inear
  {C}omplementarity {P}roblem}, Computer Science and Scientific Computing,
  Academic Press, Inc., Boston, MA, 1992.

\bibitem{ncpv1}
{\sc F.~Facchinei and J.-S. Pang}, {\em Finite-dimensional {V}ariational
  {I}nequalities and {C}omplementarity {P}roblems. {V}ol. {I}}, Springer Series
  in Operations Research, Springer-Verlag, New York, 2003.

\bibitem{ncpv2}
\leavevmode\vrule height 2pt depth -1.6pt width 23pt, {\em Finite-dimensional
  {V}ariational {I}nequalities and {C}omplementarity {P}roblems. {V}ol. {II}},
  Springer Series in Operations Research, Springer-Verlag, New York, 2003.

\bibitem{Golub}
{\sc G.~H. Golub and C.~F. Van~Loan}, {\em Matrix {C}omputations}, Johns
  Hopkins Studies in the Mathematical Sciences, Johns Hopkins University Press,
  Baltimore, MD, fourth~ed., 2013.

\bibitem{NCP}
{\sc P.~T. Harker and J.-S. Pang}, {\em Finite-dimensional variational
  inequality and nonlinear complementarity problems: a survey of theory,
  algorithms and applications}, Math. Program., 48 (1990), pp.~161--220.

\bibitem{cp}
{\sc S.~Karamardian}, {\em The complementarity problem}, Math. Program., 2
  (1972), pp.~107--129.

\bibitem{Kolda}
{\sc T.~G. Kolda and B.~W. Bader}, {\em Tensor decompositions and
  applications}, SIAM Rev., 51 (2009), pp.~455--500.

\bibitem{Lim}
{\sc L.~Lim}, {\em {Singular values and eigenvalues of tensors: A variational
  approach}}, in {IEEE CAMSAP 2005: First International Workshop on
  Computational Advances in Multi-Sensor Adaptive Processing}, {IEEE}, {2005},
  pp.~{129--132}.

\bibitem{strict}
{\sc J.~J. Mor{\'e}}, {\em Classes of functions and feasibility conditions in
  nonlinear complementarity problems}, Math. Program., 6 (1974), pp.~327--338.

\bibitem{numerical}
{\sc J.~Nocedal and S.~J. Wright}, {\em Numerical {O}ptimization},
  Springer-Verlag, 1999.

\bibitem{NCPalgorithm}
{\sc M.~A. Noor}, {\em On the nonlinear complementarity problem}, J. Math.
  Anal. Appl., 123 (1987), pp.~455--460.

\bibitem{Qiliqun}
{\sc L.~Qi}, {\em Eigenvalues of a real supersymmetric tensor}, J. Symbolic
  Comput., 40 (2005), pp.~1302--1324.

\bibitem{Copositive}
\leavevmode\vrule height 2pt depth -1.6pt width 23pt, {\em Symmetric
  nonnegative tensors and copositive tensors}, Linear Algebra Appl., 439
  (2013), pp.~228--238.

\bibitem{SQ}
{\sc Y.~Song and L.~Qi}, {\em Properties of some classes of structured
  tensors}, J Optim. Theory Appl., to appear (2014).

\end{thebibliography}
}
\end{document}